\documentclass[a4paper,12pt]{amsart}
\usepackage{a4wide}
\usepackage{eucal}
\usepackage{enumerate}
\usepackage{easymat}
\usepackage{multirow}
\usepackage{color}
\usepackage[colorlinks=true,linkcolor=blue]{hyperref}
\usepackage{bbm}
\usepackage[utf8]{inputenc}
\usepackage[T1]{fontenc}
\usepackage{verbatim}

\usepackage{latexsym}
\usepackage{amsmath, amssymb, amsfonts,amsthm}
\usepackage{dsfont}
\usepackage{mathtools}
\usepackage{marvosym}
\usepackage{mathrsfs} 
\usepackage{stmaryrd} 

\advance\textheight by 2cm

\theoremstyle{plain}
\newtheorem{thm}{Theorem}[section]

\newtheorem{cor}[thm]{Corollary}
\newtheorem{prop}[thm]{Proposition}

\theoremstyle{definition}
\newtheorem{definition}[thm]{Definition}
\newtheorem{assumption}[thm]{Assumption}
\newtheorem{example}[thm]{Example}

\theoremstyle{remark}
\newtheorem{remark}[thm]{Remark}

\newcounter{stp}

\newcommand{\E}{\mathcal{E}}
\newcommand{\T}{\mathcal{T}}
\newcommand{\Cub}{C_{ub}(\mathbb{R}_+,X)}
\newcommand{\Cb}{C_{b}(\mathbb{R}_+,X)}

\begin{document}
\title[Robustness of asymptotic properties]{Asymptotic properties of $C_0$-semigroups 
under 
perturbations}
\author{Martin Adler}
\footnote{This work was completed during a research stay at the University of Memphis. I wish to express my gratitute to Prof. Jerome Goldstein for his hospitality. }
\address{Martin Adler\\
Arbeitsbereich Funktionalanalysis \\
Mathematisches Institut\\
Auf der Morgenstelle 10 \\
D-72076 T\"{u}bingen}
\email{maad@fa.uni-tuebingen.de}

\begin{abstract}
For a given $C_0$-semigroup $(T(t))_{t\geq 0}$ we consider Staffans-Weiss perturbations $(B,C)$ of its generator as studied in \cite{ABombEngel14} and investigate the robustness of asymptotic properties of the perturbed $C_0$-semigroup $(T_{BC}(t))_{t\geq 0}$. 
As a concrete application we study the asymptotic behavior of a neutral semigroup. 
\end{abstract}
\keywords{$C_0$-semigroups, perturbation, admissibility, asymptotics, neutral equations}
\subjclass[2010]{47D06, 47A55, 34E10, 34K40}

\maketitle

\section{Introduction}
In 1953, R. Phillips \cite{Phillips53} started the investigation of qualitative properties of $C_0$-semigroups which are preserved under bounded perturbations of their generators and showed that immediate norm continuity is one of them. 
Since then, many such invariant properties (or counterexamples) for bounded and unbounded perturbations have been found, see  \cite{ChowLeiva96,KunstmannWeis01,NagPia1998,Renardy95,Voigt1994,Voigt1980}. 

In this paper we concentrate on asymptotic properties and therefore consider subspaces $\E\subset \Cub$ of functions $f\in \E$ having a certain characteristic asymptotic property. 
We start with a bounded $C_0$-semigroup $(T(t))_{t\geq 0}$ having orbits in $\E$ and call $\E$ a \textit{robust subspace} for some perturbation if the orbits of the perturbed semigroup remain in $\E$.
Moreover, the asymptotic property is said to be \textit{robust} under this perturbation. 

Such robust asymptotic properties have been investigated for \textit{Miyadera-Voigt} (see V. Casarino and S. Piazzera \cite{CasPia2001}) and \textit{Desch-Schappacher perturbations} (see L. Maniar \cite{Maniar2005}). 
We refer to \cite{BouManMou2003,BouManMou06,CasManPia2002,Maniar2004} for further publications treating similar questions.

In this note we extend such results to the class of Staffans-Weiss perturbations. 
This class was introduced by George Weiss \cite[Thms. 6.1, 7.2]{Weiss94} and Olof Staffans \cite[Sects. 7.1, 7.4]{Staffans2005} in the context of regular linear systems. 
We use the operator-theoretic approach given in \cite{ABombEngel14}. 

\section{Staffans-Weiss Perturbation Theory}
We recall the Staffans-Weiss perturbation theorem as stated in \cite[Theorem 3.1]{ABombEngel14}. 
To this aim, let $(A,D(A))$ be the generator of a $C_0$-semigroup $(T(t))_{t\geq 0}$ on the Banach space $X$. 
Consider an additional Banach space $U$ and operators $B\in \mathcal{L}(U,X_{-1}^A)$, $C\in \mathcal{L}(Z,U)$, where $X_{-1}^A$ is the extrapolation space with respect to $A$ (see \cite[III. Section 5]{EN2000}) and $Z$ is a Banach space such that $X_1^A \hookrightarrow Z \hookrightarrow X$. 
Under compatibility and admissibility assumptions on the pair $(B,C)$ it is shown that $A_{BC} := (A_{-1} + BC)_{|_X}$ with domain
\begin{align*}
	D(A_{BC}) := \{x\in Z : A_{-1} x + B C x \in X \}
\end{align*}
generates a $C_0$-semigroup $(T_{BC}(t))_{t\geq 0}$ on $X$ (see \cite[Thm. 3.1]{ABombEngel14}). 
Let us make this more precise. 

\begin{definition}
\label{compatibel}
Under the above assumptions, the triple $(A,B,C)$ 
is called \emph{compatible} if for some $\lambda\in\rho(A)$ we have
\begin{equation*}\label{bild}
	\mathrm{range}\bigl(R(\lambda,A_{-1}) B\bigr)\subset Z.
\end{equation*}
\end{definition}

\begin{definition}
\label{infinitetimeadmissibilty}
Let the triple $(A,B,C)$ be compatible and take $1\leq p < \infty$. 
\begin{enumerate}[(i)]
\item The operator $B\in\mathcal{L}(U,X_{-1}^A)$ is \emph{$p$-admissible} if there exists $t>0$ such that for all $u\in L^p(0,t;U)$
\begin{equation*}
	\int_0^{t} T_{-1}(t-s) B u(s) \,\mathrm{ds} \in X.
\end{equation*}
In this case we define for each $t > 0$ the \textit{control map} $\mathcal{B}_t \in \mathcal{L}(L^p(0,t;U),X)$ corresponding to $A$ and $B$ as in \cite[Rem. 2.2]{ABombEngel14}.

\item The operator $C\in\mathcal{L}(Z,U)$ is \emph{$p$-admissible} if there exists $M_C\geq 0$ and $t>0$ such that 
\begin{equation*}
	\int_0^t \left\| CT(s)x \right\|_U^p \mathrm{ds} \leq M_C \left\| x \right\|_X^p
\end{equation*}
for all $x\in D(A)$. 
In this case we define for each $t>0$ the \textit{observation map} $\mathcal{C}_t\in \mathcal{L}(X,L^p(0,t;U))$ corresponding to $A$ and $C$ as in \cite[Rem. 2.4]{ABombEngel14}.

\item The pair $(B,C)\in\mathcal{L}(U,X_{-1})\times\mathcal{L}(Z,U)$ 
is \emph{$p$-admissible} if there exists $M_{BC}\geq 0$ and $t>0$ such that
\begin{equation*}
	\int_0^t \left\| C \int_0^{r} T_{-1}(r-s) B u(s) \,\mathrm{ds}\right\|_U^p \mathrm{dr} 
		\leq M_{BC} \left\|u \right\|_p^p
\end{equation*}
holds for all $u\in W^{2,p}_0(0,t;U) := \{f\in W^{2,p}(0,t;U) : f(0) = f'(0) = 0\}$. 
In this case we define for each $t>0$ the \textit{input-output map} $\mathcal{F}_t \in \mathcal{L}(L^p(0,t;U))$ corresponding to $A$, $B$ and $C$ as in \cite[Rem. 2.7]{ABombEngel14}.
\end{enumerate}
\end{definition}

\begin{remark}
\label{expStability}
If the $C_0$-semigroup $(T(t))_{t\geq 0}$ is uniformly exponentially stable, it suffices to require admissibility for one $t>0$ only to obtain time independent constants $M_B$, $M_C$ and $M_{BC}$, see \cite[Rem. 2.6]{Weiss89Control}, \cite[Rem. 2.4]{Weiss89Observation}, \cite[Prop. 2.1]{Weiss89Representation}. 
\end{remark}

We now recall the class of perturbations considered in \cite{ABombEngel14}.
\begin{definition}
\label{class}
Let $(A,D(A))$ be the generator of a $C_0$-semigroup $(T(t))_{t\geq 0}$ on a Banach space $X$, $B\in \mathcal{L}(U,X^A_{-1})$ and $C\in \mathcal{L}(Z,U)$ for some Banach space $Z$ satisfying  $X^A_1\hookrightarrow Z\hookrightarrow X$.
We call the pair $(B,C)$ a \emph{Staffans-Weiss perturbation} of $A$ if for some $1\leq p < \infty$ the following conditions are satisfied:
\begin{enumerate}[(i)]
\item The triple $(A,B,C)$ is compatible.
\item The operator $B$ is $p$-admissible.
\item The operator $C$ is $p$-admissible.
\item The pair $(B,C)$ is $p$-admissible.
\item The operator $I-\mathcal{F}_t$ is invertible for some $t>0$.
\end{enumerate}
\end{definition}

For such perturbations the following result holds (see \cite{Staffans2005,Weiss94} or \cite{ABombEngel14}).
\begin{thm}
\label{ALTperturbation}
Let $(A,D(A))$ be the generator of a $C_0$-semigroup $\mathcal{T}=(T(t))_{t\geq 0}$ on a Banach space $X$. 
Assume that $(B,C)$ is a Staffans-Weiss perturbation of $A$. Then
\begin{equation}\label{eq:def-A_BC}
	A_{BC} := (A_{-1}+BC)|_{X}, \quad D(A_{BC}) := \bigl\{x\in Z:A_{-1} x + B C x \in X\bigr\}
\end{equation}
generates a $C_0$-semigroup $\mathcal{T}_{BC}=(T_{BC}(t))_{t\geq 0}$ on the Banach space $X$ satisfying
\begin{equation}
	T_{BC}(t) x = T(t)x + \int\limits_0^t T_{-1}(t-s) BC T_{BC}(s) x \mathrm{ds} \quad\text{for } x\in D(A_{BC}).
\end{equation} 
\end{thm}

\section{Robustness of asymptotic properties}
We now turn to the investigation of robustness of asymptotic properties of $C_0$-semigroups under Staffans-Weiss perturbations. 
To do so, we consider the orbits $\T x := {[}t\mapsto T(t)x{]}$ and $\T_{BC}x$ for all $x\in X$. 
For a bounded $C_0$-semigroup $(T(t))_{t\geq 0}$ the orbits $\T x$ belong to $\Cub$, the space of all bounded, uniformly continuous functions from $\mathbb{R}_+$ to $X$. 
We now look for subspaces $\E\subset \Cub$ appropriate for our purpose.  

\begin{definition}
Let $\E \subset \Cub$ be a closed subspace. 
We call $\E$ an \textit{asymptotic subspace} if for all $t\geq 0$ and $f\in \Cub$
\begin{align}
	S(t) f\in \E \quad \Longrightarrow \quad f\in \E, \label{translationbiinvariant}
\end{align}
where $(S(t))_{t\geq 0}$ denotes the left translation semigroup on $\Cub$. 
\end{definition}

\begin{remark}
The authors of \cite{BatChill99,CasPia2001} call subspaces satisfying \eqref{translationbiinvariant} \textit{translation-(bi)invariant}.
\end{remark}

We present a list of bounded $C_0$-semigroups whose orbits lead to asymptotic subspaces, see \cite[Sect. 7]{BatChill99} and \cite{CasPia2001}. 
\begin{enumerate}[(i)]
\item $(T(t))_{t\geq 0}$ is \textit{bounded}, 
i.e., there exists a constant $M\geq 1$ such that $\left\| T(t) \right\| \leq M$ for all $t\geq 0$. 
\item $(T(t))_{t\geq 0}$ is \textit{compact}, 
i.e., for all $x\in X$ the orbits $\T x$ are relatively compact in $X$.
\item $(T(t))_{t\geq 0}$ is \textit{weakly compact}, 
i.e., for all $x\in X$ the orbits $\T x$ are relatively weakly compact in $X$.
\item $(T(t))_{t\geq 0}$ is \textit{uniformly compact}, 
i.e., $\{T(\cdot +t) x: t\in \mathbb{R}_+ \} \subset C_b(\mathbb{R}_+,X)$ is relatively compact for all $x\in X$, 
where $T(\cdot +t) x := {[} s\mapsto T(s+t)x {]}\in \Cub$. 
\item $(T(t))_{t\geq 0}$ is \textit{uniformly weakly compact},
i.e., $\{T(\cdot +t) x: t\in \mathbb{R}_+ \}$ weakly relatively compact in $C_b(\mathbb{R}_+,X)$ for all $x\in X$.
\item $(T(t))_{t\geq 0}$ is \textit{strongly stable}, 
i.e., $\left\| T(t)x \right\| \rightarrow 0$ as $t\rightarrow \infty$ for all $x\in X$. 
\item $(T(t))_{t\geq 0}$ is \textit{weakly stable}, 
i.e., $\left| \langle \phi, T(t)x \rangle \right| \rightarrow 0$ as $t\rightarrow \infty$ for all $x\in X$ and $\phi \in X'$. 
\item $(T(t))_{t\geq 0}$ is \textit{mean ergodic}, i.e., for every $x\in X$ the Cesaro limit
\begin{align*}
	\lim\limits_{t\rightarrow \infty} \frac{1}{t} \int_0^t T(s)x \,\mathrm{ds}
\end{align*}
exists in $X$. 
\item $(T(t))_{t\geq 0}$ is \textit{uniformly ergodic}, 
i.e., for every $x\in X$ the Cesaro limit
\begin{align*}
	\lim_{t\rightarrow \infty} \frac{1}{t} \int_0^t T(\cdot+s) x \,\mathrm{ds}
\end{align*}
exists in $\Cub$. 
\end{enumerate}

We now express our problem in abstract form. 
Let $\E\subset \Cub$ be an asymptotic subspace and let $\mathcal{T}=(T(t))_{t\geq 0}$ be a bounded $C_0$-semigroup such that all orbits $\T x$ belong to $\E$.  
Under what conditions on the Staffans-Weiss perturbation $(B,C)$ do the orbits of the perturbed $C_0$-semigroup $\T_{BC} = (T_{BC}(t))_{t\geq 0}$ remain in $\E$? 
If this is the case, then we call $\E$ a \textit{robust subspace for $(B,C)$}. \\

In order to find appropriate perturbing operators, we strengthen the requirements on the Staffans-Weiss perturbation $(B,C)$. 
\begin{definition}
\label{SWperturbation}
Let $(A,D(A))$ be the generator of a $C_0$-semigroup $(T(t))_{t\geq 0}$ on a Banach space $X$. 
We call a Staffans-Weiss perturbation $(B,C)$ ($p$ as in Definition \ref{class}) an \emph{infinite-time Staffans-Weiss perturbation} of $A$ if 
\begin{enumerate}[(i)]
\item $B$ is \emph{infinite-time $p$-admissible}, that is, there exists $M_B\geq 0$ such that for all $t>0$ and $u\in L^p(0,t;U)$ we have
\begin{equation*}
	\int_0^{t} T_{-1}(t-s) B u(s) \,\mathrm{ds} \in X
\end{equation*}
and 
\begin{equation*}
	\left\| \int_0^{t} T_{-1}(t-s) B u(s) \,\mathrm{ds} \right\|_X \leq M_B \left\| u\right\|_{L^p(0,t;U)}, 
\end{equation*} \label{inftimeAdm}
\item $\sup_{t>0} \left\| (I-\mathcal{F}_t)^{-1} \mathcal{C}_t \right\|_{\mathcal{L}(X,L^p(0,t;U))} < \infty$. \label{invertibility}
\end{enumerate}
\end{definition}

Under these additional assumptions, our perturbation result reads as follows. 
\begin{prop}
\label{UniformBoundedness}
Let $(A,D(A))$ be the generator of a bounded $C_0$-semigroup $\mathcal{T}=(T(t))_{t\geq 0}$ and assume that $(B,C)$ is an infinite-time Staffans-Weiss perturbation of $A$. 
Then $(A_{BC}, D(A_{BC}))$ generates a bounded $C_0$-semigroup $\mathcal{T}_{BC}=(T_{BC}(t))_{t\geq 0}$ on the Banach space $X$.
\end{prop}
\begin{proof}
By Theorem 3.1 in \cite{ABombEngel14}, the perturbed semigroup $(T_{BC}(t))_{t\geq 0}$ is given by
\begin{align}
	T_{BC}(t)x = T(t)x + \mathcal{B}_t (I-\mathcal{F}_t)^{-1} \mathcal{C}_t x, \quad x\in X. \label{DefSemigroup}
\end{align}
The boundedness of $(T_{BC}(t))_{t\geq 0}$ follows since $(T(t))_{t\geq 0}$ is bounded and
\begin{align*}
	\sup_{t>0} \left\| \mathcal{B}_t (I-\mathcal{F}_t)^{-1} \mathcal{C}_t x \right\|_X 
		\leq M_B \sup_{t>0} \left\| (I-\mathcal{F}_t)^{-1} \mathcal{C}_t x \right\|_p < \infty
\end{align*}
by Definition \ref{SWperturbation} $\eqref{inftimeAdm}$ and $\eqref{invertibility}$.
\end{proof}

\begin{remark}
\begin{enumerate}[(i)]
\item Let $(A,D(A))$ be the generator of a bounded $C_0$-semigroup $(T(t))_{t\geq 0}$ on $X$. 
Proposition \ref{UniformBoundedness} states that $\E=\Cub$ is robust for all infinite-time Staffans-Weiss perturbations $(B,C)$ of $A$. 
\item In Proposition \ref{UniformBoundedness} we obtain the generator property of $A_{BC}$ without a rescaling of the original semigroup $(T(t))_{t\geq 0}$. 
Thus, it allows us to investigate the robustness of asymptotic properties under Staffans-Weiss perturbations. 
\end{enumerate}
\end{remark}

\begin{remark}
Let the $C_0$-semigroup $(T(t))_{t\geq 0}$ be uniformly exponentially stable with a compatible triple $(A,B,C)$ such that $B$, $C$ and $(B,C)$ are $p$-admissible with $1\in \rho(\mathcal{F}_t)$ for some $t>0$ and $1\leq p < \infty$. 
By Remark \ref{expStability} and \cite[Lemma 3.3]{ABombEngel14} the condition
\begin{align}
	\left\| T(t) + \mathcal{B}_t (I-\mathcal{F}_t)^{-1} \mathcal{C}_t \right\| < 1 \quad \text{for some } t>0 \label{perturbedSGexpStable}
\end{align}
implies that $1\in \rho(\mathcal{F}_{\infty})$ and thus $(B,C)$ is an infinite-time Staffans-Weiss perturbation of $A$. 
Then the perturbed $C_0$-semigroup $(T_{BC}(t))_{t\geq 0}$ remains uniformly exponentially stable by \eqref{DefSemigroup} and \cite[V. Prop. 1.7]{EN2000}. 
\end{remark}

The following theorem is our main result.

\begin{thm}
\label{ThmAsymptotik}
Let $(T(t))_{t\geq 0}$ be a bounded $C_0$-semigroup on $X$ with generator $(A,D(A))$ and let $\E$ be an asymptotic subspace such that $\T x \in \E$ for all $x\in X$. 
If $(B,C)$ is an infinite-time Staffans-Weiss perturbation of $A$, then $\E$ is a robust subspace for $(B,C)$. 
\end{thm}

In the proof we shall use techniques proposed in \cite{BatChill99} and continued in \cite{CasPia2001,Maniar2005}. 

\begin{proof}
We notice that $(I-\mathcal{F}_t)^{-1} \mathcal{C}_t x \in L^p(0,t;U)$ for all $x\in X$, $t>0$, and $p$ as in Definition \ref{class}. 
Further, ${[}t\mapsto (I-\mathcal{F}_t)^{-1} \mathcal{C}_t x (t){]} \in L^p(\mathbb{R}_+,U)$ by Definition \ref{SWperturbation} \eqref{invertibility}. 

By Formula \eqref{DefSemigroup} it suffices to show that ${[}t\mapsto \mathcal{B}_t u{]}\in \E$ for all $u\in L^p(\mathbb{R}_+,U)$. 
From the assumption on $B$ we obtain that 
\begin{align*}
	\mathbb{B}: L^p(\mathbb{R}_+,U) &\rightarrow \Cb, \\
	u &\mapsto \mathbb{B} u
\end{align*}
is a bounded operator, where $(\mathbb{B} u)(t) := \mathcal{B}_t u$. 
In fact, the strong continuity of $(\mathcal{B}_t)_{t\geq 0}$ (see \cite[Lemma 3.2]{ABombEngel14}) implies the continuity of $\mathbb{B} u$. 
The boundedness of $\mathbb{B} u$ follows from Definition \ref{SWperturbation} \eqref{inftimeAdm}, i.e.,
\begin{align*}
	\left\| \mathbb{B} u \right\|_{C_b(\mathbb{R}_+,X)} 
		= \sup_{t>0} \left\| \mathcal{B}_t u \right\| 
		\leq M_B \left\| u \right\|_p .
\end{align*}
We show that $f := \mathbb{B} \tilde{u} \in \E$ for all $\tilde{u}= \mathds{1}_{|_{(a,b)}} \otimes u$, $u\in U$, and $0\leq a < b$. 
The left translation semigroup on $\Cub$ is denoted by $(S(t))_{t\geq 0}$. 
For $t > 0$ we have the identity
\begin{align*}
	S(b) f (t) = \mathcal{B}_{t+b} \tilde{u}
		&= \int_{0}^{t+b} T_{-1}(t+b-s) B \tilde{u}(s) \;\mathrm{ds} \\
		&= \int_{a}^{b} T_{-1}(t+b-s) B u \;\mathrm{ds} \\
		&= T(t) \int_{0}^{b-a} T_{-1}(b-a-s) B u \;\mathrm{ds} \\
		&= T(t) \mathcal{B}_{b-a} \left( \mathds{1}_{|_{(0,b-a)}} \otimes u\right).
\end{align*}
Using the admissibility of $B$ we have $\mathcal{B}_{b-a} \left( \mathds{1}_{|_{(0,b-a)}} \otimes u\right)\in X$. 
Since $\T x \in \E$ for all $x\in X$, we obtain $S(b) f \in \E$. 
Thus, $\mathbb{B} \tilde{u} \in \E$ since $\E$ is an asymptotic subspace. 

Finally, we obtain $\mathbb{B} u \in \E$ for all $u\in L^p(\mathbb{R}_+,U)$ since the step functions are dense in $L^p(\mathbb{R}_+,X)$ and $\E$ is closed. 
\end{proof}

\begin{remark}
\begin{enumerate}
\item In contrast to the results by Casarino, Piazzera \cite{CasPia2001} and Maniar \cite{Maniar2005} we do not assume the subspace $\E$ to be operator invariant, i.e., 
\begin{align*}
	f\in \E \quad \Rightarrow \quad {[}t\mapsto \mathcal{M} f(t){]} \in \E \quad \forall\, \mathcal{M}\in \mathcal{L}(X).
\end{align*}
\item If one only wants that an individual orbit $\T_{BC} x$ belongs to $\E$ for some $x\in X$, it suffices to assume $\T x \in \E$ and $\mathcal{B} u_x \in \E$, where $u_x \in L^p(\mathbb{R}_+,U)$ is given by
\begin{align*}
	u_x|_{{[}0,T{]}} := (I-\mathcal{F}_T)^{-1} \mathcal{C}_T x \qquad \forall\; T>0. \\
\end{align*}
\end{enumerate}
\end{remark}

In the next corollary we show that Theorem \ref{ThmAsymptotik} generalizes Theorem 3.5 from \cite{CasPia2001} on Miyadera-Voigt perturbations. 
\begin{cor}
\label{CorCASPIA}
Let $(T(t))_{t\geq 0}$ be a bounded $C_0$-semigroup with generator $(A,D(A))$ and uniform bound $M$. 
Take $C\in \mathcal{L}(X_1,X)$ satisfying 
\begin{align}
	\int_{0}^{t} \left\| C T(s) x \right\| \,\mathrm{ds} \leq q \left\| x \right\| \label{MVBedingung}
\end{align}
for all $x\in D(A)$, some $0\leq q < 1$, and all $t>0$. 
If $\E$ is an asymptotic subspace with $\T x \in \E$ for all $x\in X$, then $\E$ is a robust subspace for $(\mathrm{Id},C)$. 
\end{cor}
\begin{proof}
We show that the assumptions of Theorem \ref{ThmAsymptotik} are satisfied. 
For arbitrary $t_0>0$ and $f\in L^1(0,t_0;X)$ we obtain
\begin{align*}
	\left\|\int_{0}^{t_0} T(t_0-s)f(s) \,\mathrm{ds}\right\| \leq M \int_{0}^{t_0} \left\|f(s)\right\| \,\mathrm{ds} 
		= M \left\| f \right\|_1. 
\end{align*}
Hence $\mathrm{Id}$ is infinite-time $1$-admissible, while condition \eqref{MVBedingung} implies the 
$1$-admissibility of $C$. 
We obtain 
\begin{align}
	\int_{0}^{t_0} \left\| C \int_{0}^{t} T(t-r) u(r) \,\mathrm{dr} \right\| \,\mathrm{dt} \leq q \left\| u \right\|_1 \label{BC-MV}
\end{align}
for all $t_0>0$ and $u\in L^1(0,t_0,U)$. 
In fact, take $\lambda \in \rho(A)$. 
For $u = \mathds{1}_{|_{{[}a,b{]}}} \otimes x$, $x\in D(A)$, $0\leq a < b \leq T$, we have
\begin{align*}
	\int_0^{t_0} \left\| C \int_{0}^{t} T(t-r) u(r) \,\mathrm{dr} \right\| \,\mathrm{dt}
		&= \int_0^{t_0} \left\| C R(\lambda,A) \int_0^t T(t-r) (\lambda-A) u(r) \mathrm{dr} \right\| \mathrm{dt} \\
		&= \int_0^{t_0} \left\| \int_0^t C T(t-r) u(r) \mathrm{dr} \right\| \mathrm{dt} \\
		&\leq \int_0^{t_0} \int_{0}^{t} \mathds{1}_{|_{{[}a,b{]}}}(r) \left\| C T(t-r) x \right\| \,\mathrm{dr} \,\mathrm{dt} \\
		&= \int_0^{t_0} \mathds{1}_{|_{{[}a,b{]}}}(r) \int_{0}^{t_0-r} \left\| C T(t) x \right\| \,\mathrm{dt} \,\mathrm{dr} \\
		&\leq q \left\| x \right\| (b-a) = q \left\| u \right\|_1 .
\end{align*}
The above estimate holds for step functions having values in $D(A)$ by linearity and we obtain \eqref{BC-MV} by the density of such functions in $L^1(0,T,X)$ for all $T>0$. 
Hence, $(\mathrm{Id},C)$ is $1$-admissible and $1\in \rho(\mathcal{F}_{t_0})$ for all $t_0>0$ with $\sup_{t_0>0} \left\| (I-\mathcal{F}_{t_0})^{-1} \right\| \leq (1-q)^{-1}$. 
For $x\in D(A)$ we have
\begin{align*}
	\sup_{t_0>0} \left\| (I-\mathcal{F}_{t_0})^{-1} \mathcal{C}_{t_0} x \right\| 
		\leq (1-q)^{-1} \sup_{t_0>0} \int_0^{t_0} \left\| C T(s) x \right\| \mathrm{ds} 
		\leq \frac{q}{1-q} \left\| x \right\|,
\end{align*}
and $(\mathrm{Id},C)$ is an infinite-time Staffans-Weiss perturbation of $A$. 
\end{proof}

Before we consider Desch-Schappacher perturbations in Proposition \ref{ExamplesBOUNDED} and show that Theorem \ref{ThmAsymptotik} relates to Section 3 in \cite{Maniar2005}, we first recall the definition of the Favard class of a $C_0$-semigroup $(T(t))_{t\geq 0}$ on $X$ given by
\begin{align*}
	F_1 := \left\{ x\in X : \sup_{t > 0} \left\| \frac{1}{t} (T(t) x - x) \right\| < \infty \right\} \subset X
\end{align*}
equipped with the norm $\left\| x \right\|_{F_1} := \sup_{t > 0} \left\| \frac{1}{t} (T(t) x - x) \right\|$, see \cite[II. Def. 5.10]{EN2000}.

We denote the Favard class associated to the extrapolated $C_0$-semigroup $(T_{-1}(t))_{t\geq 0}$ by $F_0$. 

Let $(T(t))_{t\geq 0}$ be a $C_0$-semigroup such that $\left\|T(t)\right\| \leq M e^{\omega t}$ for some $M>0$ and $\omega>\omega_0$. 
Then there exists a constant $m>0$ such that 
\begin{align}
	\left\| \int_0^t T_{-1}(t-r) f(r) \,\mathrm{dr} \right\|_X \leq m \int_0^t e^{\omega (t-r)} \left\| f(r) \right\|_{F_0} \mathrm{dr} \label{Favard}
\end{align}
for all $f\in L^1_{loc}(\mathbb{R}_+,F_0)$ and $t>0$, see \cite[Prop. 3.3]{NagSin93}.

In Proposition \ref{ExamplesBOUNDED} and Example \ref{ExTranslation} we give some elementary examples of infinite-time Staffans-Weiss perturbations. 
\begin{prop}
\label{ExamplesBOUNDED}
Let $(T(t))_{t\geq 0}$ be a uniformly exponentially stable $C_0$-semigroup with generator $(A,D(A))$, i.e., $\left\| T(t) \right\| \leq M e^{- \omega t}$ for all $t>0$ and some $\omega > 0$. 
If $B\in \mathcal{L}(X,F_0)$ satisfies $m \left\| Bx \right\|_{F_0} < \omega \left\| x \right\|$ for all $x\in X$, then $(B,\mathrm{Id})$ is an infinite-time Staffans-Weiss perturbation of $A$, where the constant $m$ is as in \eqref{Favard}. 
\end{prop}

\begin{proof}
In \cite[Thm. 4.1]{ABombEngel14} the authors show that $(B,\mathrm{Id})$ is a Staffans-Weiss perturbation of $A$ and the operator $B$ is infinite-time $1$-admissible since the semigroup $(T(t))_{t\geq 0}$ is uniformly exponentially stable, see Remark \ref{expStability}. 
It remains to show that
\begin{align*}
	\sup_{t>0} \left\| (I-\mathcal{F}_t)^{-1} {[}T(\cdot)x{]} \right\|_{L^1(0,t;X)} < \infty \qquad \forall x\in X.
\end{align*}

For all $x\in X$ and $n\in \mathbb{N}_0$ we have
\begin{align}
	\int_0^{\infty} \left\| \mathcal{F}_t^n {[}T(\cdot) x{]}(t) \right\| \mathrm{dt} \leq \left(\frac{m \left\| B \right\|}{\omega}\right)^n \frac{M}{\omega} \left\| x \right\|. \label{ManiarInfTime}
\end{align}
In fact, for $n=0$ we obtain 
\begin{align*}
	\int_0^{\infty} \left\| T(t) x \right\| \mathrm{dt} \leq \frac{M}{\omega} \left\| x \right\|.
\end{align*}
Assume that \eqref{ManiarInfTime} holds for some $n\in \mathbb{N}$. By \eqref{Favard} we have
\begin{align*}
	\int_0^{\infty} \left\| \mathcal{F}_t^{n+1} {[}T(\cdot) x{]} (t) \right\| \mathrm{dt}
		&= \int_0^{\infty} \left\| \int_0^t T_{-1}(t-r) B \mathcal{F}_r^n {[}T(\cdot) x{]}(r) \,\mathrm{dr} \right\| \mathrm{dt} \\
		&\leq m \left\| B \right\| \int_0^{\infty} \int_0^t e^{-\omega(t-r)} \left\| \mathcal{F}_r^n {[}T(\cdot) x{]}(r) \right\| \mathrm{dr} \,\mathrm{dt} \\
		&= m \left\| B \right\| \int_0^{\infty} e^{\omega r} \left\| \mathcal{F}_r^n {[}T(\cdot) x{]}(r) \right\| \int_r^{\infty} e^{-\omega t} \,\mathrm{dt} \,\mathrm{dr} \\
		&= \frac{m \left\| B \right\|}{\omega} \int_0^{\infty} \left\| \mathcal{F}_r^n {[}T(\cdot) x{]}(r) \right\| \mathrm{dr} \\
		&\leq \left(\frac{m \left\| B \right\|}{\omega}\right)^{n+1} \frac{M}{\omega} \left\| x \right\|. 
\end{align*}
Hence, for $x\in X$ we obtain
\begin{align*}
	\left\| \sum_{n\in \mathbb{N}_0} \mathcal{F}_t^n {[}T(\cdot) x{]} \right\|_1 
		\leq \sum_{n\in \mathbb{N}_0} \left\| \mathcal{F}_t^n {[}T(\cdot) x{]} \right\|_1
		\leq \frac{M}{\omega - m \left\| B \right\|} \left\| x \right\|
\end{align*}
and the assertion \eqref{ManiarInfTime} follows. 
\end{proof}

\begin{remark}
Following the proof in \cite[Sect. 3]{Maniar2005}\footnote{The author studies robustness for \textit{bounded} $C_0$-semigroups.} we obtain
\begin{align*}
	\sup_{t>0} \left\| B (I-\mathcal{F}_t)^{-1} \mathcal{C}_t \right\|_{\mathcal{L}(X,L^1(0,t;F_0))} < \infty
\end{align*}
with similar arguments as above.
Hence, under the assumptions in \cite{Maniar2005} the pair $(B,\mathrm{Id})$ need not be an infinite-time Staffans-Weiss perturbation of $(A,D(A))$ in general. 
\end{remark}

Next, we discuss boundary perturbations of translation semigroups, see \cite{Greiner87} and \cite[Sect. 4.3]{ABombEngel14}. 
\begin{example}
\label{ExTranslation}
Let $(T(t))_{t\geq 0}$ be the left translation $C_0$-semigroup on $X=L^1(\mathbb{R}_-)$ generated by
\begin{align*}
	A f = f', \quad f\in D(A) := \{f\in W^{1,1}(\mathbb{R}_-): f(0)=0\}.
\end{align*}
We denote by $A_{-1}$ the generator of the extrapolated $C_0$-semigroup $(T_{-1}(t))_{t\geq 0}$ on $X_{-1}^A$. 
For $\lambda \in \rho(A) = \{\lambda \in \mathbb{C}: \mathrm{Re}\lambda > 0\}$, we define the Dirichlet operator corresponding to $\lambda$ and $\left(\frac{d}{ds},W^{1,1}(\mathbb{R}_-)\right)$,
\begin{align*}
	D_{\lambda} = \left( {\delta_0}|_{\mathrm{ker}\left(\lambda-\frac{d}{ds}\right)} \right)^{-1}: &\quad \mathbb{C} \rightarrow L^1(\mathbb{R}_-) \\
		&\quad c \mapsto c e^{\lambda \cdot}.
\end{align*}
Let $C\in C_0(\mathbb{R}_-)'$ satisfy $\left\| C \right\| < 1$. 
Then $((\lambda - A_{-1})D_{\lambda},C)$ is an infinite-time Staffans-Weiss perturbation of $A$. 
\end{example}
\begin{proof}
We can represent the operator $C\in C_0(\mathbb{R}_-)'$ as a Riemann-Stieltjes integral
\begin{align*}
	C f = \int_{\mathbb{R}_-} f(s) \,\mathrm{d}\mu(s), \quad f\in C_0(\mathbb{R}_-),
\end{align*}
where $\mu$ is a regular complex Borel measure on $\mathbb{R}_-$ that satisfies $\left|\mu\right|(\mathbb{R}_-) = \left\| C \right\| < 1$. 
Here, $\left|\mu\right|$ is the variation of $\mu$. 

We first verify that the triple $(A,(\lambda-A_{-1}) D_{\lambda}, C)$ is compatible. 
For $c\in \mathbb{C}$ we have 
\begin{align*}
	R(\lambda,A_{-1}) (\lambda-A_{-1}) D_{\lambda} c = c e^{\lambda \cdot} \in C_0(\mathbb{R}_-). 
\end{align*}
Next, for $u\in W^{1,1}_0(0,t_0)$, $t_0>0$, we have
\begin{align*}
	\int_0^{t_0} T_{-1}(t_0-r) (\lambda - A_{-1})D_{\lambda} u(r) \,\mathrm{dr}
		&= e^{\lambda t_0} \int_0^{t_0} e^{-\lambda (t_0 - r)}  T_{-1}(t_0-r) (\lambda - A_{-1})D_{\lambda} e^{-\lambda r} u(r) \,\mathrm{dr} \\
		&= e^{\lambda t_0} \left( D_{\lambda} e^{-\lambda t_0} u(t_0) 
			- \int_0^{t_0} e^{-\lambda (t_0 - r)}  T(t_0-r) D_{\lambda} {[}e^{-\lambda r} u(r){]}' \,\mathrm{dr} \right) \\
		&= D_{\lambda} u(t_0) - \int_0^{t_0} T(t_0-r) D_{\lambda} {[}u'(r) - \lambda u(r){]} \,\mathrm{dr} \\
		&= e^{\lambda \cdot} u(t_0) + \int_{\max\{0,\cdot+t_0\}}^{t_0} \lambda u(r) e^{\lambda (\cdot + t_0 - r)} \,\mathrm{dr} \\
			&\qquad - \int_{\max\{0,\cdot+t_0\}}^{t_0} u'(r) e^{\lambda (\cdot + t_0 - r)} \mathrm{dr} \\
		&= e^{\lambda \min\{0,\cdot+t_0\}} u(\max\{0,\cdot+t_0\}) \in L^1(\mathbb{R}_-).
\end{align*}
Hence, $\left\| \mathcal{B}_{t_0} u \right\|_{L^1(\mathbb{R}_-)} \leq \left\| u \right\|_1$ for all $t_0>0$. 

The $1$-admissibility of $C$ 
follows as in \cite[Cor. 4.10]{ABombEngel14}, i.e., for $f\in D(A)$ we have
\begin{align*}
	\int_0^{\infty} \left| C T(t) f \right| \mathrm{dt}
		&= \int_0^{\infty} \left| \int_{-\infty}^{-t} f(t+s) \,\mathrm{d}\mu(s) \right| \mathrm{dt} \\
		&\leq \int_{-\infty}^0 \int_0^{-s} \left| f(t+s)\right| \mathrm{dt}\, \mathrm{d}\left|\mu\right|(s) \\
		&\leq \left| \mu \right|(\mathbb{R}_-) \left\| f \right\|_1.
\end{align*} 
Finally, using the above computations for $u\in W^{1,1}_0(\mathbb{R}_+)$, we obtain
\begin{align*}
	\left\| \mathcal{F}_{\infty} u \right\|_{L^1(\mathbb{R}_+)}
		&= \int_0^{\infty} \left| \int_{\mathbb{R}_-} e^{\lambda \min\{0,s+t\}} u(\max\{0,s+t\}) \,\mathrm{d}\mu(s) \right| \mathrm{dt} \\
		&= \int_0^{\infty} \left| \int_{-t}^0 u(s+t) \,\mathrm{d}\mu(s) \right| \mathrm{dt} 
		\leq \int_{-\infty}^0 \int_{-s}^{\infty} \left| u(s+t) \right| \mathrm{dt} \;\mathrm{d}\left|\mu\right|(s) \\
		&= \left| \mu \right|(\mathbb{R}_-) \left\| u \right\|_1.
\end{align*}
Thus, $1\in \rho(\mathcal{F}_{\infty})$ and the pair $((\lambda - A_{-1})D_{\lambda},C)$ satisfies all conditions in Definition \ref{SWperturbation}. 
\end{proof}

\section{Neutral Semigroup}
\label{Applications}
As an application of the results in Section 3 we investigate the asymptotic behavior of the neutral semigroup $(\mathfrak{T}(t))_{t\geq 0}$ (see \eqref{neutralgenerator} for its generator) associated to the following neutral equation. 
We suppose 
\begin{itemize}
\item $(A,D(A))$ to be the generator of a bounded $C_0$-semigroup $(T(t))_{t\geq 0}$ on a Banach space $X$ with uniform bound $M$, 
\item 
$\mathcal{P}\in \mathcal{L}(C({[}-1,0{]},X),X)$, 
\item $F = \delta_0 - \mathcal{K}$ for some $\mathcal{K}\in \mathcal{L}(C({[}-1,0{]},X),X)$, 
\end{itemize}
and investigate the equation
\begin{align*}
	(NE)\qquad \frac{d}{dt} F x_t = A F x_t + \mathcal{P} x_t, \qquad t\geq 0,
\end{align*}
with initial data $x(0) = y$ and $x_0(\cdot) = f(\cdot): {[}-1,0{]} \rightarrow X$ where 
we denote by $x_t: {[}-1,0{]} \rightarrow X$ the history segments given by $x_t(s) := x(t+s)$. 
For further information, see \cite{BurnsHerdmanStech83,HaddRhandi08,NagelHuy03} and references therein. 

We introduce $z(t) := F x_t: {[}-1,0{]} \rightarrow X$ in order to rewrite $(NE)$. We obtain
\begin{align*}
	\begin{cases}
		\;\frac{d}{dt} z(t) = A z(t) + \mathcal{P} x_t, &t\geq 0, \\
		\;\frac{d}{dt} x_t = \frac{d}{ds} x_t, &t\geq 0, \\
		\;z(t) = F x_t = x(t) - \mathcal{K} x_t, \quad &t\geq 0,
	\end{cases}
\end{align*}
with respective initial conditions, where we used \cite[Lemma 3.4]{BatkaiPiazzera05}. 

In \cite[Prop. 21]{HaddRhandi08} Hadd and Rhandi treat such equations and show that the system is \textit{well-posed in a weak sense} (i.e., $(NE)$ has unique \textit{generalized solutions} for any initial value $(x,f)\in \mathfrak{X}:= X\times L^1(-1,0;X)$, see \cite[Def. 17]{HaddRhandi08}), if 
\begin{align}
	\mathfrak{A} &:=	\begin{pmatrix}
						A & \mathcal{P} \\
						0 & \frac{d}{ds}
					\end{pmatrix} \label{neutralgenerator} \\
	\text{with }\;\; D(\mathfrak{A}) &:= \{ (x,f) \in D(A) \times W^{1,1}(-1,0;X) : x = f(0) - \mathcal{K} f \} \nonumber
\end{align}
is the generator of the \textit{neutral semigroup} $(\mathfrak{T}(t))_{t\geq 0}$ on $\mathfrak{X}$, where $\left( \frac{d}{ds}, W^{1,1}(-1,0;X)\right)$ denotes the first derivative. 
More generally, we investigate the generator property of $(\mathfrak{A},D(\mathfrak{A}))$ with 
\begin{align*}
	D(\mathfrak{A}) = \{ (x,f) \in D(A) \times W^{1,1}(-1,0;X) : C x = f(0) - \mathcal{K} f \}
\end{align*}
for $C \in \mathcal{L}(X)$ and then the asymptotic properties of the generated \textit{neutral semigroup} $(\mathfrak{T}(t))_{t\geq 0}$. 
We shall return to the neutral equation (with $C=\alpha \cdot \mathrm{Id}$) in Remark \ref{C=IdNeutral}. 

Our starting point is the operator
\begin{align*}
	\mathfrak{A}_0 :=	\begin{pmatrix}
								A & 0 \\
								0 & D
							\end{pmatrix}, \qquad D(\mathfrak{A}_0) := D(A) \times \{ f\in W^{1,1}(-1,0;X): f(0)=0 \},
\end{align*}
where $D := \frac{d}{ds}$. The operator $(\mathfrak{A}_0,D(\mathfrak{A}_0))$ generates the bounded $C_0$-semigroup 
\begin{align*}
	(\mathfrak{T}_0(t))_{t\geq 0} = \begin{pmatrix}
										T(t) & 0 \\ 0 & S(t)
									\end{pmatrix}_{t\geq 0}
\end{align*}
with $(S(t))_{t\geq 0}$ the nilpotent left translation semigroup on $L^1(-1,0;X)$. 
Its asymptotic properties depend essentially on the semigroup $(T(t))_{t\geq 0}$ since $(S(t))_{t\geq 0}$ is nilpotent. \\

Let the operator $L_0:X\rightarrow L^1(-1,0;X)$ be given by $L_0 x := \mathds{1} \cdot x$. 
We obtain the operator $\mathfrak{A}$ through a perturbation from $\mathfrak{A}_0$ by a pair $(\mathfrak{B},\mathfrak{C})$ as follows. 
\begin{prop}
\label{DefinitionBC}
Let $\mathfrak{A}_{0,-1}$ be the generator of the extrapolated semigroup $(\mathfrak{T}_{0,-1}(t))_{t\geq 0}$ on the extrapolation space $\mathfrak{X}_{-1}^{\mathfrak{A}_0}$. Define the operators
\begin{align*}
	\mathfrak{B} &:=	\begin{pmatrix}
							I & 0 \\
							0 & - D_{-1} L_0
						\end{pmatrix} : X \times X \rightarrow 
												\mathcal{X}_{-1}^{\mathfrak{A}_0}, \\
	\mathfrak{C} &:= 	\begin{pmatrix}
							0 & \mathcal{P} \\
							C & \mathcal{K}
						\end{pmatrix} : D(\mathfrak{C}) \subset \mathcal{X} \rightarrow X \times X,
\end{align*}
where $D(\mathfrak{C}) := X \times W^{1,1}(-1,0;X)$. 
Then $\mathfrak{A} = \left( \mathfrak{A}_{0,-1} + \mathfrak{B} \mathfrak{C} \right)|_{\mathfrak{X}}$.
\end{prop}
\begin{proof}
The domains $D(\mathfrak{A})$ and $D\left((\mathfrak{A}_{0,-1} + \mathfrak{B} \mathfrak{C})|_{\mathfrak{X}}\right)$ coincide since
\begin{align*}
	D\left( \left( \mathfrak{A}_{0,-1} + \mathfrak{B} \mathfrak{C} \right)|_{\mathfrak{X}} \right) 
		&:= \left\{ (x,f)\in D(\mathfrak{C}) : A_{-1} x + \mathcal{P} f \in X,\; D_{-1} (f - L_0 C x - L_0 \mathcal{K} f)\in L^1(-1,0;X) \right\} \\
		&= \left\{ (x,f)\in D(A)\times W^{1,1}(-1,0;X) : {[}f - L_0 C x - L_0 \mathcal{K} f{]}(0) = 0 \right\} \\
		&= \left\{ (x,f)\in D(A)\times W^{1,1}(-1,0;X) : f(0) - C x - \mathcal{K} f = 0 \right\}. 
\end{align*}
Take $(x,f)\in D(\mathfrak{A})$, then
\begin{align*}
	(\mathfrak{A}_{0,-1} + \mathfrak{B} \mathfrak{C}) \begin{pmatrix} x\\f \end{pmatrix}
		&= \begin{pmatrix} A_{-1} x + \mathcal{P} f \\ D (f - L_0 C x - L_0 \mathcal{K} f) \end{pmatrix} \\
		&= \begin{pmatrix} A_{-1} x + \mathcal{P} f \\ \frac{d}{ds} f - \frac{d}{ds} L_0 (C x - \mathcal{K} f) \end{pmatrix}
		= \begin{pmatrix} A_{-1} x + \mathcal{P} f \\ \frac{d}{ds} f \end{pmatrix} = \mathfrak{A} \begin{pmatrix} x\\f \end{pmatrix}
\end{align*}
since the range of the operator $L_0$ is contained in $\mathrm{ker}\,\frac{d}{ds}$. 
\end{proof}

The operator $(\mathfrak{A}, D(\mathfrak{A}))$ is indeed a generator under suitable assumptions on $\mathcal{K}$ and $\mathcal{P}$. 
\begin{assumption}
\label{noMassAssumption}
Let $\mu$ and $\nu:{[}-1,0{]}\rightarrow \mathcal{L}(X)$ be of bounded variation. 
We assume that the operators $\mathcal{K}$, $\mathcal{P}\in \mathcal{L}(C({[}-1,0{]},X),X)$ are given by the Riemann-Stieltjes integrals
\begin{align*}
	\mathcal{P} f &= \int_{-1}^0 f(r) \;\mathrm{d}\mu(r), \\
	\mathcal{K} g &= \int_{-1}^0 g(r) \;\mathrm{d}\nu(r), \quad f,\; g\in C({[}-1,0{]},X),
\end{align*}
and have \textit{no mass in $0$}, i.e., for every $\epsilon>0$ there exists $\delta > 0$ such that 
\begin{align*}
	\left\| \mathcal{P} f \right\|_X,\; \left\| \mathcal{K} f \right\|_X \leq \epsilon \left\| f \right\|_{\infty}
\end{align*}
for every $f\in C({[}-1,0{]},X)$ satisfying $\mathrm{supp}\,f \subset {[}-\delta,0{]}$. 
\end{assumption}

\begin{remark}
We define the variation $\left| \mu \right|$ of the measure $\mu:{[}-1,0{]} \rightarrow \mathcal{L}(X)$ to be
\begin{align*}
	\left| \mu \right| (A) := \sup_{\mathcal{Z}} \sum_{E\in \mathcal{Z}} \left\| \mu(E) \right\|, \quad A\subset {[}-1,0{]} \;\text{measurable},
\end{align*}
where the supremum is taken over all partitions $\mathcal{Z}$ into finitely many disjoint, measurable subsets of $A$. 
Let $\mathcal{P}\in \mathcal{L}(C({[}-1,0{]},X),X)$ satisfy Assumption \ref{noMassAssumption}. 
Then $\left| \mu \right| {[}-t,0{]} \stackrel{t\searrow 0}{\longrightarrow} 0$.
\end{remark}

\begin{thm}
\label{neutralGeneration}
Let $(\mathfrak{A}_0,D(\mathfrak{A}_0))$ be as above and assume that $\mathcal{P}$, $\mathcal{K}\in \mathcal{L}(C(-1,0;X),X)$ satisfy Assumption \ref{noMassAssumption}. 
Then $(\mathfrak{A},D(\mathfrak{A}))$ is the generator of a $C_0$-semigroup on $\mathfrak{X}$ for all $C\in \mathcal{L}(X)$. 
\end{thm}
\begin{proof}
We show that $(\mathfrak{B},\mathfrak{C})$ as above is a Staffans-Weiss perturbation of $\mathfrak{A}_0$ (see Definition \ref{class}). 
The triple $(\mathfrak{A}_0,\mathfrak{B},\mathfrak{C})$ is compatible since for $x$, $y\in X$ we have
\begin{align*}
	R(\lambda,\mathfrak{A}_{0,-1}) \mathfrak{B} \begin{pmatrix} x\\y \end{pmatrix} = 	\begin{pmatrix}
																							R(\lambda,A) x & 0\\
																							0 & L_0 y - \lambda R(\lambda,D) L_0 y 
																						\end{pmatrix} \in X \times W^{1,1}(-1,0;X). 
\end{align*} 
The following relies on computations performed in \cite[Cor. 4.10]{ABombEngel14}, see Example \ref{ExTranslation} as well. 
For arbitrary $t_0>0$ and $u_1$, $u_2\in W^{1,1}(0,t_0;X)$ we obtain
\begin{align*}
	\left\| \int_0^{t_0} \mathfrak{T}_{0,-1}(t_0-r) \mathfrak{B} \begin{pmatrix} u_1 \\ u_2 \end{pmatrix}(r) \mathrm{dr} \right\|
		&= \left\| \int_0^{t_0} T(t_0-r) u_1(r) \mathrm{dr} \right\| \\
			&\qquad+ \left\| \int_0^{t_0} S_{-1}(t_0-r) (- D_{-1} L_0) u_2(r) \mathrm{dr} \right\| \\
		&\leq M \left\| u_1 \right\|_1 + \left\| u_2 \right\|_1 =: M \left\| \begin{pmatrix} u_1 \\ u_2 \end{pmatrix}\right\|_{L^1(-1,0;X\times X)}.
\end{align*}
Thus, $\mathfrak{B}$ is infinite-time 1-admissible. 
Choose $t_0 > 0$. 
For all $(x,f)\in D(\mathfrak{A}_0)$ we obtain
\begin{align}
	\int_0^{t_0} \left\| \begin{pmatrix} 0&\mathcal{P}\\C&\mathcal{K}\end{pmatrix} \mathfrak{T}_0(s) \begin{pmatrix} x\\f \end{pmatrix} \right\| \mathrm{ds}
		&\leq \int_0^{t_0} \left\| \mathcal{P} S(s) f \right\| \mathrm{ds} \nonumber \\
			&\qquad+ \int_0^{t_0} \left\| \mathcal{K} S(s) f \right\| \mathrm{ds} 
			+ \int_0^{t_0} \left\| C T(s) x \right\| \mathrm{ds} \nonumber \\
		&\leq \int_0^1 \left| \int_{-1}^{-s} f(r+s) \,\mathrm{d}\mu(r) \right| \mathrm{ds} \nonumber \\
			&\qquad+ \int_0^1 \left| \int_{-1}^{-s} f(r+s) \,\mathrm{d}\nu(r) \right| \mathrm{ds}
			+ \int_0^{t_0} \left\| C T(s) f \right\| \mathrm{ds} \label{Cadmissible} \\
		&\leq \int_{-1}^0 \int_0^{-r} \left| f(r+s) \right| \mathrm{ds} \,\mathrm{d}\left|\mu\right|(r) \nonumber \\
			&\qquad+ \int_{-1}^0 \int_0^{-r} \left| f(r+s) \right| \mathrm{ds} \,\mathrm{d}\left|\nu\right|(r)
			+ t_0 \left\| C \right\| M \left\| x \right\| \nonumber \\
		&\leq \left\| f \right\|_1 \left( \int_{-1}^0 \mathrm{d}\left|\mu\right|(r) 
			+ \int_{-1}^0 \mathrm{d}\left|\nu\right|(r) \right)
			+ t_0 \left\| C \right\| M \left\| x \right\| \nonumber \\
		&= (\left\| \mu \right\| + \left\| \nu \right\|) \left\| f \right\|_1
			+ t_0 \left\| C \right\| M \left\| x \right\|, \nonumber 
\end{align}
where $\left\| \mu \right\| := \left| \mu \right| {[}-1,0{]}$ (and $\left\|\nu\right\|$ respectively). 
Hence, the operator $\mathfrak{C}$ is 1-admissible. 
For $0<t_0\leq 1$ and $u_1$, $u_2\in W^{2,1}_0(0,t_0;X)$
\begin{align*}
	\int_0^{t_0} \left\| \begin{pmatrix} 0&\mathcal{P}\\C&\mathcal{K}\end{pmatrix} \int_0^t \mathfrak{T}_{0,-1}(t-r) \mathfrak{B} \begin{pmatrix} u_1(r)\\u_2(r) \end{pmatrix} \mathrm{dr}\right\| &\mathrm{dt} \\
		&\hspace{-3cm}\leq \int_0^{t_0} \left\| \mathcal{P} \int_0^t S_{-1}(t-r) (-D_{-1}L_0) u_2(r) \mathrm{dr}\right\| \mathrm{dt} \\
			&\hspace{-3cm}\qquad + \int_0^{t_0} \left\| \mathcal{K} \int_0^t S_{-1}(t-r) (-D_{-1}L_0) u_2(r) \mathrm{dr}\right\| \mathrm{dt} \\
			&\hspace{-3cm}\qquad + \int_0^{t_0} \left\| C \int_0^t T(t-r) u_1(r) \mathrm{dr}\right\| \mathrm{dt} \\
		&\hspace{-3cm}\leq (\left|\mu\right| {[}-t_0,0{]} + \left|\nu\right| {[}-t_0,0{]})\left\| u_2 \right\|_1 + t_0 \left\| C \right\| M \left\| u_1 \right\|_1,
\end{align*}
see \cite[Cor. 4.10]{ABombEngel14} for analogous computations. 
Hence, the pair $(\mathfrak{B},\mathfrak{C})$ is 1-admissible and $1\in \rho(\mathcal{F}_{t_0})$ for some $t_0$ sufficiently small by Assumption \ref{noMassAssumption}. 
Theorem \ref{ALTperturbation} yields the assertion. 
\end{proof}

In order to obtain the following robustness result we have to make sure that $(\mathfrak{B},\mathfrak{C})$ is an infinite-time Staffans-Weiss perturbation of $\mathfrak{A}_0$. 
\begin{prop}
\label{neutralEqu}
Let $(\mathfrak{A}_0,D(\mathfrak{A}_0))$ be the generator of the bounded $C_0$-semigroup $(\mathfrak{T}_0(t))_{t\geq 0}$. Assume that
\begin{enumerate}[(i)]
\item ${[}t\mapsto T(t) x{]} \in \E(X)$ for some asymptotic subspace $\E(X) \subset \Cub$ and all $x\in X$,
\item $\mathcal{P}$ and $\mathcal{K}\in \mathcal{L}(C({[}-1,0{]},X),X)$ satisfy Assumption \ref{noMassAssumption},
\item $(\mathfrak{B},\mathfrak{C})$ satisfies condition $\eqref{invertibility}$ in Definition \ref{SWperturbation}.
\end{enumerate}
Then ${[}t\mapsto \mathfrak{T}(t) z{]} \in \E(\mathfrak{X})$ for all $z\in \mathfrak{X}$, where $\E(\mathfrak{X}) \subset C_{ub}(\mathbb{R}_+,\mathfrak{X})$ is defined analogously to $\E(X)$. 
\end{prop}
\begin{proof}
By Theorem \ref{neutralGeneration}, the pair $(\mathfrak{B},\mathfrak{C})$ is a Staffans-Weiss perturbation of $\mathfrak{A}_0$ and $\mathfrak{B}$ is infinite-time $1$-admissible. 
Since we assume the condition $\eqref{invertibility}$ in Definition \ref{SWperturbation} to hold, Theorem \ref{ThmAsymptotik} gives the assertion. 
\end{proof}

We present a situation in which the conditions of Proposition \ref{neutralEqu} are satisfied. 
\begin{cor}
\label{Situation}
Let $(\mathfrak{A}_0,D(\mathfrak{A}_0))$ be the generator of the $C_0$-semigroup $(\mathfrak{T}_0(t))_{t\geq 0}$, $(\mathfrak{B},\mathfrak{C})$ as in Proposition \ref{DefinitionBC}. Assume, in addition, that
\begin{enumerate}[(i)]
\item ${[}t\mapsto T(t) x{]} \in \E(X)$ for some asymptotic subspace $\E(X) \subset \Cub$ and all $x\in X$,
\item $\mathcal{P} = p \delta_{-1}$, $\mathcal{K} = k \delta_{-1}$ with $p$, $k\in \mathcal{L}(X)$ satisfying $\left\| p \right\| + \left\| k \right\| < 1$, 
\item $C\in \mathcal{L}(X)$ such that $\int_0^t \left\| C T(s) x \right\| \mathrm{ds} \leq q \left\| x \right\|$ for all $x\in D(A)$, $t>0$ and some $0\leq q < 1$. \label{CinfAdmissible}
\end{enumerate}
Then ${[}t\mapsto \mathfrak{T}(t)z{]}\in \E(\mathfrak{X})$ for all $z\in \mathfrak{X}$.
\end{cor}
\begin{proof}
By Proposition \ref{neutralEqu} it remains to show that condition $\eqref{invertibility}$ in Definition \ref{SWperturbation} is satisfied. 
For $u\in W^{2,1}_0(\mathbb{R}_+,X)$ we have
\begin{align*}
	\int_0^{\infty} \left\| \mathcal{P} \int_0^t S_{-1}(t-r) (-D_{-1}L_0) u(r) \mathrm{dr} \right\| \mathrm{dt}
		&\leq \left\| p \right\| \int_0^{\infty} \left\| \delta_{-1} \int_0^t S_{-1}(t-r) (-D_{-1}L_0) u(r) \mathrm{dr} \right\| \mathrm{dt} \\
		&= \left\| p \right\| \int_0^{\infty} \left\| u(t) - \delta_{-1} \int_0^t S(t-r) L_0 u'(r) \mathrm{dr} \right\| \mathrm{dt} \\
		&= \left\| p \right\| \int_0^{\infty} \left\| u(t) - \int_{\max\{0,t-1\}}^t u'(r) \mathrm{dr} \right\| \mathrm{dt} \\
		&= \left\| p \right\| \int_1^{\infty} \left\| u(t-1) \right\| \mathrm{dt} = \left\| p \right\| \left\| u \right\|_1.
\end{align*}
We show that $1\in \rho(\mathcal{F}_{\infty})$. 
For $u_1$, $u_2\in W^{2,1}_0(\mathbb{R}_+,X)$ we obtain
\begin{align}
	\int_0^{\infty} \left\| \mathfrak{C} \int_0^t \mathfrak{T}_{0,-1}(t-r) \mathfrak{B} \begin{pmatrix} u_1(r) \\ u_2(r) \end{pmatrix} \mathrm{dr} \right\| \mathrm{dt} 
		&\leq \int_0^{\infty} \left\| \mathcal{P} \int_0^t S_{-1}(t-r) (-D_{-1}L_0) u_2(r) \mathrm{dr}\right\| \mathrm{dt} \nonumber\\
			&\quad + \int_0^{\infty} \left\| \mathcal{K} \int_0^t S_{-1}(t-r) (-D_{-1}L_0) u_2(r) \mathrm{dr}\right\| \mathrm{dt} \nonumber\\
			&\quad + \int_0^{\infty} \left\| C \int_0^t T(t-r) u_1(r) \mathrm{dr}\right\| \mathrm{dt} \label{thirdSummand}\\
		&\leq (\left\| p \right\| + \left\| k \right\|) \int_1^{\infty} \left\| u_2(t - 1) \right\| \mathrm{dt} 
			+ q \left\| u_1 \right\|_1 \nonumber\\
		&\leq (\left\| p \right\| + \left\| k \right\|) \left\| u_2 \right\|_1 
			+ q \left\| u_1 \right\|_1 \nonumber
\end{align}
where we estimated \eqref{thirdSummand} as in the proof of Corollary \ref{CorCASPIA}. 
Hence, $\left\| \mathcal{F}_{\infty} \right\| < 1$. 

For $t>0$, we denote by $\mathfrak{C}_t$ the observation map corresponding to $\mathfrak{A}_0$ and $\mathfrak{C}$. 
Using assumption \eqref{CinfAdmissible} and repeating the estimate \eqref{Cadmissible} in Theorem \ref{neutralGeneration}, we obtain 
\begin{align*}
	\sup_{t>0} \left\| (1-\mathcal{F}_t)^{-1} \mathfrak{C}_t \begin{pmatrix} x\\f \end{pmatrix} \right\| 
		&\leq (1 - \left\| \mathcal{F}_{\infty} \right\|)^{-1} \sup_{t>0} \int_0^t \left\| \begin{pmatrix} 0&\mathcal{P}\\C&\mathcal{K}\end{pmatrix} \mathfrak{T}_0(s) \begin{pmatrix} x\\f \end{pmatrix} \right\| \mathrm{ds} \\
		&\leq (1 - \left\| \mathcal{F}_{\infty} \right\|)^{-1} \sup_{t>0} \left( \int_0^t \left\| C T(s) x \right\| \mathrm{ds} + (\left\| k \right\| + \left\| p \right\|) \int_0^t \left\| \delta_{-1} S(s) f \right\| \mathrm{ds} \right) \\
		&\leq (1 - \left\| \mathcal{F}_{\infty} \right\|)^{-1} \left( q \left\| x \right\| + (\left\| k \right\| + \left\| p \right\|) \int_0^1 \left\| f(s-1) \right\| \mathrm{ds} \right) \\
		&= (1 - \left\| \mathcal{F}_{\infty} \right\|)^{-1} ( q \left\| x \right\| + (\left\| k \right\| + \left\| p \right\|) \left\| f \right\|_1) \\
		&< (1 - \left\| \mathcal{F}_{\infty} \right\|)^{-1} \left\| \begin{pmatrix} x \\ f \end{pmatrix} \right\|
\end{align*}
for all $(x,f) \in D(\mathcal{A}_0)$.
\end{proof}

\begin{example}
\label{C=IdNeutral}
Let $(A,D(A))$ be the generator of a uniformly exponentially stable $C_0$-semigroup $(T(t))_{t\geq 0}$ with $\left\| T(t) \right\| \leq M e^{- \omega t}$ and $k$, $p\in \mathcal{L}(X)$. 

We consider the neutral equation
\begin{align*}
	\frac{d}{dt} {[} x(t) - k x(t-1){]} &= A {[} x(t) - k x(t-1){]} + \alpha \cdot p x(t-1), \quad t\geq 0, \quad \alpha > 0. 
\end{align*}
with initial data $x_0 = f$ and $x(0) = y$. 
Generalized wellposedness of the neutral equation corresponds to the generator property of the operator matrix 
\begin{align*}
	\mathfrak{A} &= \begin{pmatrix} A & \alpha p \delta_{-1} \\ 0 & \frac{d}{ds} \end{pmatrix} \\
	D(\mathfrak{A}) &= \{ (x,f) \in D(A) \times W^{1,1}(-1,0;X) : x = f(0) - k f(-1) \}
\end{align*}
defined in \eqref{neutralgenerator}, see \cite{HaddRhandi08}. 
Using the isomorphism $\mathfrak{S}_{\alpha} \begin{pmatrix} x \\ f \end{pmatrix} = \begin{pmatrix} x \\ \alpha f \end{pmatrix}$ on $\mathfrak{X}$ we consider the operator matrix
\begin{align*}
	\tilde{\mathfrak{A}} &= \begin{pmatrix} A & p \delta_{-1} \\ 0 & \frac{d}{ds} \end{pmatrix} \\
	D(\tilde{\mathfrak{A}}) &= \{ (x,f) \in D(A) \times W^{1,1}(-1,0;X) : \alpha x = f(0) - k f(-1) \}
\end{align*}
given by $\tilde{\mathfrak{A}} = \mathfrak{S}_{\alpha} \mathfrak{A} \mathfrak{S}_{\alpha}^{-1}$.

Choose operators $k$ and $p$ with $\left\| k \right\| + \left\| p \right\| < 1$ and $\alpha > 0$ such that $M \alpha < \omega$. 
The assumptions in Corollary \ref{Situation} are satisfied (see also Proposition \ref{ExamplesBOUNDED} for $C = \alpha \cdot \mathrm{Id}$). 
Hence, $(\tilde{\mathfrak{A}},D(\tilde{\mathfrak{A}}))$ generates a $C_0$-semigroup $(\tilde{\mathfrak{T}}(t))_{t\geq 0}$ and the orbits satisfy ${[}t\mapsto \tilde{\mathfrak{T}}(t) (x,f){]} \in C_0(\mathbb{R}_+,\mathfrak{X})$ for all $(x,f) \in \mathfrak{X}$ since the semigroup $(T(t))_{t\geq 0}$ is strongly stable. 

Finally, we conclude that $(\mathfrak{A},D(\mathfrak{A}))$ generates a $C_0$-semigroup $(\mathfrak{T}(t))_{t\geq 0}$ having orbits in $C_0(\mathbb{R}_+,\mathfrak{X})$ since 
\begin{align*}
	\mathfrak{T}(t) \begin{pmatrix} x \\f \end{pmatrix} = \mathfrak{S}_{\alpha}^{-1} \tilde{\mathfrak{T}}(t) \begin{pmatrix} x \\ \alpha f \end{pmatrix} \stackrel{t\rightarrow \infty}{\longrightarrow} 0
\end{align*}
for all $(x,f) \in \mathfrak{X}$. 
\end{example}

\end{document}